\documentclass[a4paper,12pt]{article}

\usepackage{amsfonts}
\usepackage{amscd,color}
\usepackage{amsmath,amsfonts,amssymb,amscd}
\usepackage{indentfirst,graphicx,epsfig}
\usepackage{graphicx}
\input{epsf}
\usepackage{graphicx}
\usepackage{epstopdf}
\usepackage{caption}

\newtheorem{thm}{Theorem}[section]

\newtheorem{lem}[thm]{Lemma}

\newtheorem{cor}[thm]{Corollary}

\newenvironment {proof} {\noindent{\em Proof.}}{\hspace*{\fill}$\Box$\par\vspace{4mm}}
\newcommand{\ml}{l\kern-0.55mm\char39\kern-0.3mm}

\title{\textbf{The rainbow vertex-disconnection in graphs\footnote{Supported by NSFC No.11871034, 11531011 and NSFQH No.2017-ZJ-790.}}}
\author{{\small Xuqing Bai$^1$, You Chen$^1$, Ping Li$^1$, Xueliang Li$^{1,2}$, Yindi Weng$^1$ } \\
{\small  $^1$Center for Combinatorics and LPMC}\\
{\small Nankai University, Tianjin 300071, China}\\
{\small Email: baixuqing0@163.com, chen\_you@163.com}\\
{\small Email: wjlpqdxs@163.com, lxl@nankai.edu.cn, 1033174075@qq.com}\\
\small $^2$School of Mathematics and Statistics, Qinghai Normal University\\
\small Xining, Qinghai 810008, China\\
}
\date{}
\begin{document}
\maketitle

\begin{abstract}
Let $G$ be a nontrivial connected and vertex-colored graph. A subset $X$ of the vertex set of $G$ is called rainbow if any two vertices in $X$ have distinct colors. The graph $G$ is called \emph{rainbow vertex-disconnected} if for any two vertices $x$ and $y$ of $G$, there exists a vertex subset $S$ of $G$ such that when $x$ and $y$ are nonadjacent, $S$ is rainbow and $x$ and $y$ belong to different components of $G-S$; whereas when $x$ and $y$ are adjacent, $S+x$ or $S+y$ is rainbow and $x$ and $y$ belong to different components of $(G-xy)-S$. For a connected graph $G$, the \emph{rainbow vertex-disconnection number} of $G$, denoted by $rvd(G)$, is the minimum number of colors that are needed to make $G$ rainbow vertex-disconnected.

In this paper, we characterize all graphs of order $n$ with rainbow vertex-disconnection number $k$ for $k\in\{1,2,n\}$, and determine the rainbow vertex-disconnection numbers of some special graphs.
Moreover, we study the extremal problems on the number of edges of a connected graph $G$ with order $n$ and $rvd(G)=k$ for given integers $k$ and $n$ with $1\leq k\leq n$.


\end{abstract}

\section{Introduction}

All graphs considered in this paper are simple, finite and undirected. Let $G=(V(G), E(G))$ be a nontrivial connected graph with vertex set $V(G)$ and edge set $E(G)$. The $order$ of $G$ is denoted
by $n=|V(G)|$. For a vertex $v\in V(G)$, the \emph{open neighborhood} of $v$ is the set $N(v)=\{u\in V(G) | uv\in E(G)\}$ and
$d(v)=|N(v)|$ is the \emph{degree} of $v$, and the \emph{closed neighborhood} of $v$ is the set $N[v]=N(v)\cup \{v\}$. The minimum and maximum degree of $G$ are denoted by $\delta(G)$ and $\Delta(G)$, respectively. Denote by $P_n$ a path on $n$ vertices. For a subset $S$ of $V(G)$, we use $G[S]$ to denote the subgraph of $G$ \emph{induced by} $S$. Let $V_1$, $V_2$ be two disjoint vertex subsets of $G$.
We denote the set of edges between $V_1$ and $V_2$ in $G$ by $E(V_1,V_2)$. We follow \cite{BM} for graph theoretical notation and terminology not defined here.

The concept of rainbow connection coloring was introduced by Chartrand et al. \cite{GG} in 2008. A rainbow path is a path whose edges are colored pairwise differently.  An edge-coloring of a graph $G$ is a rainbow connection coloring if any two vertices of $G$ are connected by a rainbow path. The rainbow connection number of a connected graph $G$, denoted by $rc(G)$, is the minimum number of colors that ensures $G$ has a rainbow connection coloring. There are a large number of papers about the rainbow connection coloring of graphs. Rainbow vertex-connection was proposed by
Krivelevich and Yuster \cite{KR} in $2010$. For more details about the rainbow vertex-connection, we refer to \cite{LSh2} and survey papers and book \cite{LSS, LS1, LS2}.

As we know that there are two ways to study the connectivity of a graph, one way is by using paths and the other is by using cuts. So, it is natural to consider the rainbow edge-cuts and rainbow vertex-cuts for the rainbow connectivity of graphs.

In \cite{GC}, Chartrand et al. first studied the rainbow edge-cuts by introducing the concept of rainbow disconnection of graphs.
Let $G$ be a nontrivial connected and edge-colored graph. An \emph{edge-cut} of $G$ is a set $R$ of edges of $G$ such that $G-R$ is disconnected. If, moreover, any two edges in $R$ have different colors, then $R$ is called a \emph{rainbow cut}. A rainbow cut $R$ is called a
$u$-$v$ \emph{rainbow cut} if the vertices $u$ and $v$
belong to different components of $G-R$. An edge-coloring of $G$ is called a rainbow disconnection coloring if for every two distinct vertices $u$ and $v$ of $G$, there exists a $u$-$v$ rainbow cut in $G$, separating them. The \emph{rainbow disconnection number} $rd(G)$ of $G$
is the minimum number of colors required by a rainbow disconnection coloring of $G$.

In order to study the rainbow vertex-cut, we introduce the concept of rainbow vertex-disconnection number in this paper. For a connected and vertex-colored graph $G$, let $x$ and $y$ be two vertices of $G$.
If $x$ and $y$ are nonadjacent, then an $x$-$y$ \emph{vertex-cut} is a subset $S$ of $V(G)$ such that $x$ and $y$ belong to different components of $G-S$. If $x$ and $y$ are adjacent, then an $x$-$y$ \emph{vertex-cut} is a subset $S$ of $V(G)$ such that $x$ and $y$ belong to different components of $(G-xy)-S$. A vertex subset $S$ of $G$ is \emph{rainbow} if no two vertices of $S$ have the same color. An $x$-$y$ \emph{rainbow vertex-cut} is an $x$-$y$ vertex-cut $S$ such that if $x$ and $y$ are nonadjacent, then $S$ is rainbow; if $x$ and $y$ are adjacent, then $S+x$ or $S+y$ is rainbow.

A vertex-colored graph $G$ is called \emph{rainbow vertex-disconnected} if for any two vertices $x$ and $y$ of $G$, there exists an $x$-$y$ rainbow vertex-cut. In this case, the vertex-coloring $c$ is called a \emph{rainbow vertex-disconnection coloring} of $G$. For a connected graph $G$, the \emph{rainbow vertex-disconnection number}
of $G$, denoted by $rvd(G)$, is the minimum number of colors that are needed to make $G$ rainbow vertex-disconnected. A rainbow vertex-disconnection coloring with $rvd(G)$ colors is called an
$rvd$-\emph{coloring} of $G$.

As is well-known, graphs can model a wide variety of practical problems and applications in a simple and understandable way. Assigning colors to the vertices or edges of graphs can further improve this ability. The rainbow vertex-disconnection coloring could also be applied to solve many practical problems. Next, we give two examples of applications.

The rainbow vertex-disconnection coloring can model frequency assignment problem for signal towers. The signal towers can transmit and receive information. Furthermore, each tower is equipped with a signal interceptor. In order to prevent the transmission of information, we want to capture the information between any two towers and feedback the interception position. In order to solve this problem, we translate
it into a coloring problem as follows: Each signal tower $X$ is represented by a vertex, also denoted by $X$. If two signal towers can receive information from each other, then we say that they adjacent and add an edge between the two corresponding vertices. The resulting graph is denoted by $G$, and we assign a color to each vertex based on the frequency emitted by the tower. Suppose that we want to intercept the information from tower $A$ to tower $B$. If vertices $A$ and $B$ representing towers $A, B$ are nonadjacent, then the towers in the corresponding $A$-$B$ vertex-cut of $G$ open their signal interception devices. If vertices $A$ and $B$ representing towers $A,B$ are adjacent, in addition to turning on the signal interception devices in the corresponding $A$-$B$ vertex-cut of $G$, we also need to turn on the device of tower $B$. To determine the location of the tower which intercepts the information, the $A$-$B$ vertex-cut in $G$ needs to be rainbow; whereas if the vertices representing towers $A,B$ are adjacent, the vertex set consisting of the vertices in the $A$-$B$ vertex-cut of $G$ and vertex $B$ needs to be rainbow. This coloring can be relaxed further. If the vertices $A, B$ are adjacent, then we only need that the color of one of the vertices $A$ and $B$ is different from the colors of the vertices in the $A$-$B$ vertex-cut. Given a rainbow vertex set for any two vertices $A$ and $B$ in advance (if the vertices $A, B$ are nonadjacent, it refers to an $A$-$B$ rainbow vertex-cut in $G$; if the vertices $A, B$ are adjacent, it refers to the set consisting of a $A$-$B$ rainbow vertex-cut of $G$ and one of the vertices $A$ and $B$ whose color is different from the colors of vertices in the $A$-$B$ vertex-cut). It is used to determine the interception position corresponding to the frequency emitted by the tower. If it is intercepted by tower $B$, then tower $B$ sends out the frequency corresponding to the color of the endpoint in the given rainbow vertex set of $A$ and $B$. Since frequencies are expensive, it is hoped that the number of frequencies is as small as possible. Then the minimum number of frequencies required in the frequency assignment problem for signal towers is precisely the rainbow vertex-disconnection number of the corresponding graph.

Another example is as follows. In the circulation of goods, we want to prevent some things from happening, such as delivering confidential letters, smuggling drugs and trading in wildlife. We need to intercept these goods in the cities which are passed by goods. When some city intercepts the goods successfully, the city could feedback the location by transmitting signals (with some special frequency) to the other cities. For the sake of solving this practical problem, we denote each city by a vertex. We assign an edge between two vertices if the corresponding cities are connected by a transporting road, and assign a color to each vertex based on the frequency emitted by city. Assume that the goods is transported from city $A$ to city $B$ (the corresponding vertices are also denoted by $A$ and $B$). For intercepting goods, we consider $A$-$B$ vertex-cut (if vertices $A$ and $B$ are adjacent, we also need to intercept the goods in city $B$). To feedback the location of city which intercepts goods, the $A$-$B$ vertex-cut need to be rainbow to denote the different locations (if vertices $A$ and $B$ are adjacent, then the color of vertex $B$ need to differ from the colors of the vertices in the vertex-cut). We can relax the coloring further. If vertices $A$ and $B$ are adjacent, then we only need that the color of vertex $A$ or $B$ is different from the colors of the vertices in the vertex-cut. Then the minimum number of frequencies for cities required
in this problem is precisely the rainbow vertex-disconnection number of the corresponding graph.

Let $x$ and $y$ be two vertices of a graph $G$. The local connectivity $\kappa_G(x,y)$ of nonadjacent $x$ and $y$ in $G$ is the minimum number of vertices of $G$ separating $x$ from $y$ in $G$. If $x$ and $y$ are adjacent vertices in $G$, the local connectivity $\kappa_G(x,y)$ of $x$ and $y$ in $G$ is defined as $\kappa_{G-xy}(x,y)+1$. The \emph{connectivity} $\kappa(G)$ of $G$ is the minimum number of vertices whose removal results in a disconnected graph or a trivial graph. The \emph{upper connectivity} $\kappa^+(G)$ of $G$ is the upper bound of the function $\kappa_G(x,y)$.

This paper is organized as follows. In Section $2$, we provide some useful lemmas that will be used in later discussion, and we also characterize the graphs having rainbow vertex-disconnection number
$1$, $2$ and $n$, respectively. In Section $3$, we give the rainbow vertex-disconnection numbers of wheel graphs and complete multipartite graphs. In Section $4$, we determine the minimum size (number of edges) of a connected graph $G$ of order $n$ with $rvd(G)=k$ for given integers $k$ and $n$ with $1\leq k\leq n$. However, we can only give a range of the maximum size of a connected graph $G$ of order $n$ with $rvd(G)=k$
for given integers $k$ and $n$ with $1\leq k\leq n$. Further efforts should be made to get the exact value of the maximum size.

\section{Preliminaries}

At first, we state some fundamental results about the rainbow vertex-disconnection number of graphs, which will be used in the sequel.

\begin{lem}\label{rvdsubgraph}
If $G$ is a nontrivial connected graph and $H$ is a connected subgraph of $G$, then $rvd(H)\leq rvd(G)$.
\end{lem}

\noindent $Proof.$ Suppose that $c$ is an $rvd$-coloring of $G$ and $H$ is a connected subgraph of $G$. Let $c'$ be a coloring that is obtained by restricting $c$ to $H$. Let $x$ and $y$ be two vertices of $H$ and $S$ be an $x$-$y$ rainbow vertex-cut of $G$. Then $S'=S\cap V(H)$ is an $x$-$y$ rainbow vertex-cut in $H$; otherwise, if there exists an $x$-$y$ path $P$ with length at least $2$ in $H-S'$, then $P$ is also in $G-S$, a contradiction. Thus, $c'$ is a rainbow vertex-disconnection coloring of $H$ and then $rvd(H)\leq rvd(G)$.  \hfill$\Box$\\

A $block$ of a graph $G$ is a maximal connected subgraph of $G$ that contains no cut vertices. So, a block of $G$ is a cut edge of $G$ or a $2$-connected subgraph of $G$ with at least three vertices. The \emph{block decomposition} of $G$ is the set of blocks of $G$.

\begin{lem}\label{rvdblock}
Let $G$ be a nontrivial connected graph, and let $B$ be a block of $G$ such that $rvd(B)$ is maximum among all blocks of $G$. Then $rvd(G)=rvd(B)$.
\end{lem}

\noindent $Proof.$ Let $G$ be a nontrivial connected graph. Let $\{B_{1},B_{2},\cdots,B_{t}\}$ be the block decomposition of $G$, and let $k=\max\{rvd(B_{i})| 1\leq i \leq t\}$. If $G$ has no cut vertex, then $G=B_{1}$ and the result follows. Next, we assume that $G$ has at least one cut vertex. Since each block is a connected subgraph of $G$,
$rvd(G)\geq k$ by Lemma \ref{rvdsubgraph}.

Let $c_{i}$ be an $rvd$-coloring of $B_{i}$. Let $H$ be a connected graph consisting of some blocks of $G$. Let $B_{i}$ $(1\leq i\leq t)$ be the block having a vertex in common with $H$, where $B_i$ is the subgraph of $G$ but not of $H$. Suppose $v$ is the common vertex of $B_{i}$ and $H$.
We define an $exchange$ $operation$ on $B_i$ as follows: If $c_{H}(v)=c_{i}(v)$, we do nothing. If $c_{H}(v)\neq c_{i}(v)$, without loss of generality, we may assume that $c_{H}(v)=1$ and $c_{i}(v)=2$.
We assign color $1$ to the vertices of $B_{i}$ that were colored with $2$, and assign color $2$ to the vertices of $B_{i}$ that were colored with $1$.

First, we take a block, say $B_1$, and let $G_1=B_1$. Then we find a block $B(\in \{B_{2},\cdots,B_{t}\})$ which has a vertex in common with graph $G_i$ ($1\leq i\leq t-1$) and add it to $G_i$ by doing exchange operation on $B$. Denote the resulting graph by $G_{i+1}$. Repeatedly, we have $G_t=G$ and get a vertex-coloring $c'$ of $G$ with $k$ colors.

Let $x$ and $y$ be two vertices of $G$. If there exists a block, say $B_{i}$, which contains both $x$ and $y$, then any $x$-$y$ rainbow
vertex-cut in $B_{i}$ with the coloring $c'_{i}$ is an $x$-$y$ rainbow
vertex-cut in $G$. If $x$ and $y$ are in different blocks, then there is exactly one $x$-$y$ internally disjoint path, say $P$, in $G$ and the path $P$ contains at least one cut vertex, say $w$. Then vertex $w$ is an $x$-$y$ rainbow vertex-cut in $G$. Hence, $rvd(G)\leq k$.  \hfill$\Box$\\

\begin{lem}\label{rvddifcolor}
Let $G$ be a nontrivial connected graph, and let $u$ and $v$ be two vertices of $G$ having at least two common neighbors. Then $u$ and $v$ receive different colors in any $rvd$-coloring of $G$.
\end{lem}
\begin{proof}
Assume that $u,v$ have two common neighbors $x,y$. Then $uxv$, $uyv$ are two internally disjoint paths between $u$ and $v$. So, $xuy$, $xvy$ are two internally disjoint paths between $x$ and $y$. Thus, $u,v$ should be assigned different colors in any $rvd$-coloring of $G$.
\end{proof}

The following result is an immediate consequence of Lemma \ref{rvddifcolor}.

\begin{cor}\label{rvdcomplete}
For an integer $n\geq 2$, $$rvd(K_{n})=\left\{
\begin{array}{lcl}
n-1,       &      & {if~n=2,3},\\
n,         &      & {if~n\geq 4}.
\end{array} \right .$$
\end{cor}

\begin{thm}\label{rvdlocalconn}
Let $G$ be a nontrivial connected graph of order $n$. Then
$\kappa(G) \leq \kappa^+(G) \leq rvd(G) \leq n$.
\end{thm}

\begin{proof}
Obviously, the upper bound holds. For the lower bound, let $x,y$ be any two vertices of $G$. Assume that $S$ is an $x$-$y$ rainbow vertex-cut.
There are $\kappa_{G}(x,y)$ internally disjoint paths between $x$ and $y$ in $G$. If $x$ and $y$ are nonadjacent, then $rvd(G)\geq |S|\geq \kappa_{G}(x,y)$. If $x$ and $y$ are adjacent, then $S+x$ (or $S+y$) is rainbow. So, $rvd(G)\geq |S+x|=|S|+1\geq \kappa_{G}(x,y)$.
Thus, $rvd(G)\geq \kappa^+(G)\geq \kappa(G)$.
\end{proof}

From the above one can see that for a nontrivial connected graph $G$ of order $n$, $1\leq rvd(G)\leq n$. We now characterize all graphs $G$ for which $rvd(G)$ attains the lower $1$ or upper bound $n$.

\begin{thm}\label{rvd1}
Let $G$ be a nontrivial connected graph. Then $rvd(G)=1$ if and only if $G$ is a tree.
\end{thm}

\noindent $Proof.$ Assume, to the contrary, that $G$ contains a cycle $C$. Let $x,y$ be two vertices of $C$. Then $\kappa_{G}(x,y)\geq 2$. So, $rvd(G)\geq \kappa_{G}(x,y)\geq 2$ by Theorem \ref{rvdlocalconn},
a contradiction. \hfill$\Box$\\

\begin{lem}\label{rvdcycle}
If $C_{n}$ is a cycle of order $n\geq 3$, then $rvd(C_{n})=2$.
\end{lem}

\begin{proof}
If $n=3$, then $rvd(C_3)=rvd(K_3)=2$ by Corollary \ref{rvdcomplete}. Now consider $n\geq 4$. Since $\kappa(C_{n})=2$, it follows
from Theorem \ref{rvdlocalconn} that $rvd(G)\geq 2$. Suppose
$C_{n}=v_{1}v_{2}\cdots v_{n}v_{1}$. Let $c$ be a coloring of $C_{n}$ such that $c(v_1)=c(v_2)=1$ and $c(v_i)=2$ $(i\in\{3,4,\cdots,n\})$.
Let $x$ and $y$ be two vertices of $C_{n}$. If $x$ and $y$ are adjacent, then there is exactly one path $P$ with length more than two between $x$ and $y$ in $C_n$. Since $n\geq 4$, we choose a vertex $u$ on $P$ with color different from $c(x)$. Obviously, $u$ is an $x$-$y$ vertex-cut and the vertex set $\{u,x\}$ is rainbow. So, $u$ is an $x$-$y$ rainbow vertex-cut. If $x$ and $y$ are nonadjacent, then there are two $x$-$y$
paths in $C_{n}$. Since $n\geq 4$, the two paths must respectively contain an internal vertex $u$ with color $1$ and an internal vertex $v$ with color $2$. Then the vertex set $\{u,v\}$ is an $x$-$y$ rainbow vertex-cut. Thus, $c$ is a rainbow vertex-disconnection coloring of $C_{n}$ using two colors, and so $rvd(C_{n})=2$.
\end{proof}

\begin{lem}\cite{GC} \label{rvd2lem}
A $2$-connected graph $G$ is a cycle if and only if for every two vertices $u$ and $v$ of $G$, there are exactly two internally disjoint $u$-$v$ paths in $G$.
\end{lem}

\begin{thm}\label{rvd2}
Let $G$ be a nontrivial connected graph. Then $rvd(G)=2$ if and only if each block of $G$ is either a $K_2$ or a cycle and at least one block of $G$ is a cycle.
\end{thm}

\noindent $Proof.$ Let $G$ be a nontrivial connected graph. If each block of $G$ is either a $K_2$ or a cycle and at least one block of $G$ is a cycle, then $rvd(G)=2$ by Lemmas \ref{rvdblock} and \ref{rvdcycle}.

Now we verify the converse. Assume, to the contrary, that there exists at least one block which is neither a $K_2$ nor a cycle or all the blocks of $G$ are $K_2$. In the former case, by Lemma \ref{rvd2lem} there exist  two vertices $x$ and $y$ of $G$ for which $G$ contains at least three internally disjoint $x$-$y$ paths. So, $rvd(G)\geq \kappa_{G}(x,y)\geq 3$ by Theorem \ref{rvdlocalconn}, a contradiction. As for the latter case, if all the blocks of $G$ are $K_2$, then $G$ is a tree. Then, $rvd(G)=1$ from Theorem \ref{rvd1}, a contradiction.
\hfill$\Box$\\

\begin{thm}\label{rvdn}
Let $G$ be a nontrivial connected graph of order $n$. Then $rvd(G)=n$ if and only if any two vertices of $G$ have at least two common neighbors.
\end{thm}

\begin{proof}
Let $rvd(G)=n$. Assume, to the contrary, that there exist two vertices $u$ and $v$ of $G$ which have at most one common neighbor. Let $c$ be a vertex-coloring of $G$ that assigns color $1$ to $u$, $v$ and colors $2,3,\cdots,n-1$ to the remaining vertices of $G$. We claim that $c$ is a rainbow vertex-disconnection coloring; otherwise, there exist two vertices $u'$ and $v'$ such that any $u'$-$v'$ vertex-cut has at least two vertices with the same color. So, any $u'$-$v'$ rainbow vertex-cut must contain vertices $u,v$. Thus, there are two internally disjoint
paths $u'uv'$ and $u'vv'$ in $G$. The vertices $u',v'$ are two common neighbors of $u$ and $v$. Thus, $rvd(G)\leq n-1$, a contradiction.

For the converse, since there are at least two common neighbors for any two vertices of $G$, the colors of vertices in $G$ are pairwise different by Lemma \ref{rvddifcolor}. Therefore, $rvd(G)=n$.
\end{proof}

\begin{cor}
Let $G$ be a nontrivial connected graph of order $n$. If there is exactly one pair of vertices which do not have two common neighbors, then $rvd(G)=n-1$.
\end{cor}
\begin{proof}
Let $u$, $v$ be the only one pair of vertices of $G$ which do not have two common neighbors. Since any two vertices in $V(G)\setminus\{u\}$ have at least two common neighbors, the colors of vertices in $V(G)\setminus\{u\}$ are pairwise different by Lemma \ref{rvddifcolor}.
So, $rvd(G)\geq n-1$. From Theorem \ref{rvdn}, $rvd(G)\leq n-1$.
\end{proof}

At the end of this section we present a result related to the girth
of a graph. Recall that for a graph $G$, the length of a shortest cycle of $G$ is called the \emph{girth} of $G$ (the girth of an acyclic graph is zero). The following is an upper bound for $rvd(G)$ in terms of the girth of $G$.

\begin{thm}\label{rvdgirth}
Let $G$ be a nontrivial connected graph of order $n$ and girth $g$ with $g\geq 4$. Then $rvd(G)\leq n-g+2$.
\end{thm}

\begin{proof}
Let $C_g$ denote a shortest cycle of $G$ where $C_g=v_1v_2\cdots v_gv_1$. Define a vertex-coloring $c$: $V(G)\rightarrow [n-g+2]$ of $G$ as follows. Let $c(v_1)=c(v_2)=1$, $c(v_i)=2$ $(i\in\{3,4,\cdots,g\})$.
Color the remaining vertices of $G$ with distinct colors using
$3,4,\cdots,n-g+2$. Let $x$ and $y$ be two vertices of $G$. Assume that $x$ and $y$ are both in $C_g$. There exists an $x$-$y$ rainbow vertex-cut of $C_g$ by Lemma \ref{rvdcycle}. We denote it by $S$. Since the vertices in $V(G)\setminus V(C_g)$ have different colors, $V(G)\setminus V(C_g)\cup S$ is an $x$-$y$ rainbow vertex-cut of $G$.

Assume that vertices $x$ and $y$ are not both in $C_g$, say $x\notin V(C_g)$. Then $N(x)$ is a rainbow subset; otherwise, $\{v_1,v_2\}\subset N(x)$ or there exist two vertices $v_{i}$, $v_{j}$ $(i,j\in\{3,4,\cdots,g\})$ with color $2$ which are both the neighbors of $x$. Considering the former, since $v_1$ and $v_2$ are adjacent, there
exists a triangle $xv_1v_2x$, a contradiction. As for the latter, since the length of $v_{i}v_{i+1}\cdots v_{j}$ in $C_g$ is less than $g-2$, the length of cycle $xv_{i}v_{i+1}\cdots v_{j}x$ is less than $g$, a contradiction. Then $N(x)\setminus\{y\}$ is an $x$-$y$ rainbow vertex-cut. So, $c$ is a rainbow vertex-disconnection coloring.
\end{proof}

According to Theorem \ref{rvdgirth} and $rvd(P_3)=1$, we can get the following corollary. Since $rvd(K_{2,n-2})=n-2$, the upper bound
for triangle-free graphs is tight.

\begin{cor}\label{c1}
Let $G$ be a nontrivial connected triangle-free graph of order $n\geq3$. Then $rvd(G)\leq n-2$.
\end{cor}

\section{$rvd$-values for some special graphs }

We present some special graphs $G$ satisfying $rvd(G)=\kappa(G)$ and $rvd(G)=n$, respectively. Firstly, we consider the rainbow vertex-disconnection number of a wheel graph.

\begin{thm}\label{p3}
If $W_{n}=C_{n}\vee K_{1}$ is the wheel of order $n+1\geq 5$, then
$$rvd(W_{n})=\left\{
\begin{array}{lcl}
3,       &      & {if~4\mid n},\\
4,         &      & {if~4\nmid n}.
\end{array} \right .$$.
\end{thm}

\begin{proof}
Suppose that $C_{n}$=$v_{1}v_{2}\cdots v_{n}v_{1}$ and $w$ is the copy
of $K_{1}$. We distinguish the following cases.

\textbf{Case 1.} $4\mid n$.

Since $\kappa(W_{n})=3$, it follows from Theorem \ref{rvdlocalconn} that $rvd(W_n)\geq 3$. Let $c$ be a vertex-coloring of $W_{n}$ such that $c(v_{i})=1$ ($i \equiv 1,2~(mod~4)$) and $c(v_{i})=2$
($i \equiv 0,3~(mod~4)$) and $c(w)=3$. Let $x$ and $y$ be two vertices of $W_{n}$. Then one of $x$ and $y$ belongs to $C_{n}$, say $x\in C_n$.
Since $4\mid n$, $N(x)$ is a rainbow subset. The vertex set $N(x)\setminus\{y\}$ is an $x$-$y$ rainbow vertex-cut of $W_{n}$.
Thus, $c$ is a rainbow vertex-disconnection coloring of $W_{n}$ using three colors, and so $rvd(W_{n})\leq 3$.

\textbf{Case 2.} $4\nmid n$.

Let $c$ be an $rvd$-coloring of $W_n$. Assume that the number of colors in $C_n$ is $2$. Without loss of generality, let $c(v_1)$=1. Since $v_1$ and $v_3$ have two common neighbors $v_2$, $w$, the colors of $v_1$ and $v_3$ are different by Lemma \ref{rvddifcolor}. So, $c(v_3)=2$. Similarly, we color $v_5,v_7,\cdots$ alternately by color $1,2$. Finally,  we obtain $c(v_{n-1})=c(v_1)=1$, which is a contradiction by Lemma \ref{rvddifcolor}. So, the number of colors in $C_n$ is at least $3$.
Since $w$ and $v_i$ ($i\in [n]$) have two common neighbors, the colors of $w$ and $v_i$ are different by Lemma \ref{rvddifcolor}. So, $rvd(G)\geq 4$.

Define a vertex-coloring $c: V(G)\rightarrow [4]$ such that $c(v_{i})=1$ ($i \equiv 1,2~(mod~4),i\in [n-2]$), $c(v_{i})=2$ ($i \equiv 0,3~(mod~4),i\in [n-2]$), $c(v_{n-1})=c(v_n)=3$ and $c(w)=4$.
Let $x$ and $y$ be two vertices of $G$. Then one of $x$ and $y$ belongs to $C_{n}$, say $x\in V(C_n)$. Then the colors of vertices in $N(x)$ are distinct. The vertex set $N(x)\setminus\{y\}$ is an $x$-$y$ rainbow vertex-cut of $W_{n}$. Thus, $rvd(W_{n})\leq 4$.
\end{proof}

From the above, we see that $\kappa(W_{n})=rvd(W_{n})$ if $4\mid n$. Next, we determine the rainbow vertex-disconnection number of a complete multipartite graph. Furthermore, we know that there are graphs $G$ with $rvd(G)=n$.

\begin{thm}\label{t1}
Let $G=K_{n_{1},n_{2},\ldots,n_{k}}$ be a complete $k$-partite graph of order $n$, where $k\geq 2$, $1\leq n_{1}\leq n_{2} \leq\cdots \leq n_{k}$ and $n_{k}\geq 2$. Then
$$rvd(K_{n_{1},n_{2},\ldots,n_{k}})=\left\{
\begin{array}{lcl}
n,       &      & {if~k\geq 4~or~k=3,~n_3\geq n_2\geq n_1\geq2},\\
n-n_{k-1},         &      & {if~k=3,~n_1=1~or ~k=2,~n_2\geq n_1\geq 2},\\
1,         &      & {if ~k=2~and~n_1=1 }.
\end{array} \right .$$
\end{thm}

\noindent $Proof.$ Let $V_{1},V_{2},\cdots,V_{k}$ be the
vertex-partition sets of $G$ with $|V_{i}|=n_{i}$ where $i\in[k]$. We distinguish the following cases to proceed the proof.

\textbf{Case 1}. $k\geq 4$ or $k=3$, $n_3\geq n_2\geq n_1\geq2$.

Let $u$ and $v$ be any two vertices of $G$. Assume $u\in V_i$ and $v\in V_j$ ($i,j\in [k]$). If $k\geq 4$ or $k=3$, $n_3\geq n_2\geq n_1\geq2$, then $|V(G)\setminus \{V_i\cup V_j\}|\geq2$. So, $u$ and $v$ have at least two common neighbors. From Theorem \ref{rvdn}, $rvd(G)=n$.

\textbf{Case 2}. $k=3$, $n_1=1$ or $k=2$, $n_2\geq n_1\geq 2$.

If $k=3$, $n_1=1$, then choose $u\in V_1$ and $v\in V_2$. From Theorem \ref{rvdlocalconn}, $rvd(G)\geq \kappa(u,v)=n_3+1=n-n_2$. Define a vertex-coloring $c: V(G)\rightarrow [n-n_2]$ such that $c(V_3)=\{1,2,\cdots,n_3\}$, $c(V_2)=\{1,2,\cdots,n_2\}$ and
$c(V_1)=\{n_{3}+1\}$. Then $c$ is a coloring using $n_{3}+1=n-n_2$ colors. Let $x$ and $y$ be any two vertices of $G$. Then one of $x$ and  $y$ belongs to $V_2\cup V_3$, say $x\in V_2\cup V_3$. Then $N(x)$ is a rainbow subset. So, $N(x)\setminus\{y\}$ is an $x$-$y$ rainbow vertex-cut. Thus, $rvd(G)\leq n-n_2$.

If $k=2, n_2\geq n_1\geq 2$, then choose two vertices $u$ and $v$ of $V_1$. From Theorem \ref{rvdlocalconn}, $rvd(G)\geq \kappa_{G}(u,v)=n_2=n-n_1$. Define a vertex-coloring $c': V(G)\rightarrow [n-n_1]$ such that $c'(V_2)=\{1,2,\cdots,n_2\}$ and $c'(V_1)=\{1,2,\cdots,n_1\}$. Let $x$ and $y$ be any two vertices of $G$. Then $N(x)$ is a rainbow subset. So, $N(x)\setminus\{y\}$ is an $x$-$y$ rainbow vertex-cut. Thus, $rvd(G)\leq n-n_1$.

\textbf{Case 3}. $k=2$ and $n_1=1$.

Obviously, the graph $G$ is a tree. From Theorem \ref{rvd1}, $rvd(G)=1$.\hfill$\Box$\\

\section{Extremal problems }

In this section, we first investigate the following problem:

For a given pair $k$, $n$ of positive integers with $1\leq k\leq n$, what is the minimum possible size of a connected graph $G$ of order $n$ such that the rainbow vertex-disconnection number of $G$ is $k$ ?

To solve this problem, we first present some useful lemmas.
\begin{lem}\label{rved-dele-c}
Let $G$ be a connected graph with $\delta(G)\geq 3$. Then there exists a cycle $C$ such that $G-V(C)$ is connected.
\end{lem}

\begin{proof}
Suppose that for any cycle $C$ of $G$, $G-V(C)$ is disconnected. Choose a cycle $C_0$ such that $G-V(C_0)$ has a connected component $G_1$ with the maximum order. Let $G_2$ be another component. We denote the ends of $E(G_i,C_0)$ in $C_0$ by $S_i$ $(i=1,2)$. Let $x_0\in S_1$. Now we distinguish two cases.

\textbf{Case 1}. $|S_2|=1$.

Since $\delta(G_2)\geq 2$, there exists a cycle $C'$ in $G_2$ such that $G-V(C')$ has a component containing $G_1$ and $C_0$, a contradiction.

\textbf{Case 2}. $|S_2|\geq2$.

{\em Subcase $2.1$}. For $|S_2\setminus \{x_0\}|\geq 2$, we assume $\{x_1,x_2\}\in S_2$ and $x_1y_1, x_2y_2\in E(G_2,C_0)$.
There exists a path $Q_2$ between $y_1$ and $y_2$. We note that there is an $x_1$-$x_2$ path $Q_1$ in $C_0$ which does not go through $x_0$.
Choose a cycle $C'=x_1Q_1x_2y_2Q_2y_1x_1$. Then $G-V(C')$ has a component containing $C_0$ and $G_1$, a contradiction.

{\em Subcase $2.2$}. For $|S_2\setminus \{x_0\}|=1$, we assume $S_2=\{x_0,x_1\}$. When $|N(x_1)\cap V(G_2)|\geq 2$, we assume $z_1,z_2\in N(x_1)\cap V(G_2)$.
Then there is a $z_1$-$z_2$ path $Q$ in $G_2$. Choose a cycle $C'=x_1z_1Qz_2x_1$. Then $G-V(C')$ has a component containing $x_0$ and $G_1$, a contradiction.
When $|N(x_1)\cap V(G_2)|=1$, we assume $N(x_1)\cap V(G_2)=\{z_1\}$. For any vertex $v\in G_2\setminus \{z_1\}$, the degree of $v$ in $G_2$ is at least $2$. Since a tree has at least two leaves, there is a cycle $C'$ in $G_2$. Then $G-V(C')$ has a component containing $C_0$ and $G_1$, a contradiction.
\end{proof}

\begin{cor}\label{rvd-cyclecon}
Let $G$ be a connected graph with $\delta(G)\geq 3$.
Then there exists a cycle $C$ such that $G-E(C)$ is connected.
\end{cor}
\begin{proof}
By Lemma \ref{rved-dele-c}, we choose a minimum cycle, say $C$, such that $G-V(C)$ is connected. Then $C$ has no chord; otherwise we can choose a smaller cycle $C'$ such that $G-V(C')$ is connected. Therefore, $G-E(C)$ is connected.
\end{proof}

For any two vertices $x$ and $y$ of $G$, let $S_{G}(x,y)$ be a rainbow vertex-cut of $x$ and $y$ in $G$.
Let $\bar{S}_{G}(x,y)$ be a rainbow vertex set such that if $x,y$ are adjacent, then $\bar{S}_{G}(x,y)=S_{G}(x,y)+x$ or $S_{G}(x,y)+y$ is rainbow; if $x,y$ are nonadjacent, then $\bar{S}_{G}(x,y)=S_{G}(x,y)$ is rainbow. In order to prove that a vertex-coloring of $G$ is a rainbow vertex-disconnection coloring, we only need to find $S_{G}(x,y)$ or $\bar{S}_{G}(x,y)$ for any two vertices $x,y$ of $G$.

\begin{lem}\label{rvd-min1}
For integers $k$ and $n$ with $1\leq k\leq n-1$,
the minimum size of a connected graph $G$ of order
$n$ with $rvd(G)=k$ is $n+k-2$.
\end{lem}

\begin{proof}
First, we show that if the size of a connected
graph $G$ of order $n$ is $n+k-2$, then $rvd(G)\leq k$.
By induction on $k$. For $k=1$, the result is true by
Theorem \ref{rvd1}. Suppose the result holds for
$2\leq k\leq n-2$. Let $G$ be a connected
graph of order $n$ and size $n+(k+1)-2=n+k-1$.
We show that $rvd(G)\leq k+1$.
Now we proceed the proof by distinguishing the following three cases.

\textbf{Case 1}. $\delta(G)\geq 3$.

Let $|C|=\ell$ $(\geq 3)$ and $G'=G-E(C)$.
Then by Corollary \ref{rvd-cyclecon}, $G'$ is a connected graph with $|V(G')|=n$ and $|E(G')|=n+k-1-\ell$. By the induction hypothesis, we have $rvd(G')\leq k+1-\ell$.
Suppose the coloring $c'$ of $G'$ is a rainbow vertex-disconnection coloring using $k+1-\ell$ colors.
We now extend the coloring $c'$ of $G'$ to a coloring $c$ of $G$ by assigning $c(x)=c'(x)$ for
$x\in V(G)\setminus V(C)$ and assigning $\ell$ distinct new colors to the vertices of $C$.
We can verify that the coloring $c$ is a rainbow vertex-disconnection coloring of $G$.
For any two vertices $x,y\in V(G)$,
we have $S_{G}(x,y)=S_{G'}(x,y)\cup (V(C)\setminus \{x,y\})$. Hence, $rvd(G)\leq k+1$.

\textbf{Case 2}. There exists at least one vertex in $G$ with degree $2$.

{\em Subcase $2.1.$} There exists a vertex $u$ in $G$ which
is not a cut vertex and $d(u)=2$.

Let $N(u)=\{w,w'\}$
and $G'$ be a graph obtained by removing the edge $uw$ from $G$. Since
the size of $G'$ is $n+k-2$, we get $rvd(G')\leq k$
by the induction hypothesis.
Suppose that the coloring $c'$ of $G'$ is a rainbow vertex-disconnection coloring using colors from $[k]$.
If $c'(w)\neq c'(w')$, then we extend the coloring $c'$ of $G'$ to a coloring $c$ of $G$ as follows.
Let $c(u)=k+1$ and $c(x)=c'(x)$ for $x\in V(G)\setminus\{u\}$.
We can verify that the coloring $c$ is a rainbow vertex-disconnection coloring of $G$.
For two vertices $x, y\in V(G)\setminus \{u\}$, we find
$\bar{S}_{G}(x,y)=\bar{S}_{G'}(x,y)\cup \{u\}$.
For the case that one vertex is $u$ and the other is $p\in V(G)\setminus\{u\}$,
we have $\bar{S}_{G}(u,p)=\{w,w'\}$.
If $c(w)= c(w')$, then we extend the coloring $c'$ of $G'$ to a coloring $c$ of $G$ as follows.
Let $c(u)=c(w)=k+1$ and $c(x)=c'(x)$ for $x\in V(G)\setminus\{u, w\}$. We can verify that the coloring $c$ is
a rainbow vertex-disconnection coloring of $G$.
For any two vertices $x,y\in V(G)\setminus\{u,w'\}$,
if $w\notin \bar{S}_{G'}(x,y)$, then $\bar{S}_{G}(x,y)=\bar{S}_{G'}(x,y)\cup \{u\}$;
otherwise, $\bar{S}_{G}(x,y)=\bar{S}_{G'}(x,y)\cup\{w'\}$.
For the case that one of the vertices is $w'$ and the other vertex is $p\in V(G)\setminus\{u,w,w'\}$,
if $w\notin \bar{S}_{G'}(w,p)$, then $\bar{S}_{G}(w,p)=\bar{S}_{G'}(w,p)\cup\{u\}$;
otherwise, $\bar{S}_{G}(w,p)=\bar{S}_{G'}(w,p)$.
For the case that one vertex is $w$ and the other is $w'$,
if $w$ and $w'$ are adjacent, then $\bar{S}_{G}(w,w')=\bar{S}_{G'}(w,w')\setminus\{w\}\cup\{w',u\}$;
otherwise, $\bar{S}_{G}(w,w')=\bar{S}_{G'}(w,w')\cup\{u\}$.
For the case that one vertex is $u$ and the other is $p\in V(G)\setminus\{u\}$,
$\bar{S}_{G}(u,p)=\{w,w'\}$.

{\em Subcase $2.2.$} All the vertices with degree $2$ are cut vertices.

Suppose that there exist $q$ vertices with degree $2$.
We contract all the vertices with degree $2$ of $G$, namely,
we contract $q$ edges of $G$. Denote the resulting graph by $G'$.
Then $|V(G')|=n'=n-q$, $|E(G')|=n+k+1-2-q=n'+k+1-2$
and $\delta(G')\geq 3$.
It follows that $rvd(G')\leq k+1$ from Case 1. Suppose that the coloring $c'$ of $G'$ is a rainbow vertex-disconnection coloring using colors from $[k+1]$.
We now extend the coloring $c'$ of $G'$ to a coloring
$c$ of $G$ by assigning $c(x)=c'(x)$ for the
vertices with degree at least 3 and assigning
color 1 to the vertices with degree $2$.
We can verify that this coloring $c$ is a rainbow vertex-disconnection coloring of $G$.
For any two vertices $x, y$, if $d(x),d(y)\geq 3$,
we have $\bar{S}_{G}(x,y)=\bar{S}_{G'}(x,y)$;
otherwise, there exists at least one vertex with
degree $2$. Suppose $d(x)=2$. Then there exists only one path $P$ between
$x$ and $y$. We denote the neighbor of $x$ in $P$ by $z$.
Then we have $\bar{S}_{G}(x,y)=\{z\}$.
Hence, $rvd(G)\leq k+1$.

\textbf{Case 3}. $\delta(G)=1$ and no vertex is of degree 2.

Delete all the pendent vertices and pendent trees from $G$
and denote the resulting graph by $G'$.
Suppose that we delete $z$ vertices which are denoted by the set $Z$.
Then we have $\delta(G')\geq 2$ and
$|V(G')|=n'=n-z$ and $|E(G')|=n+k+1-2-z=n'+k+1-2$.
It follows from Cases 1 and 2 that $rvd(G)\leq k+1$.
Suppose the coloring $c'$ of $G'$ is a rainbow vertex-disconnection coloring using colors from $[k+1]$.
We now extend the coloring $c'$ of $G'$ to a coloring
$c$ of $G$ by assigning $c(x)=c'(x)$ for
$x\in V(G)\setminus Z $ and assigning color $1$ to remaining vertices.
We can verify that this coloring $c$ is a rainbow
vertex-disconnection coloring of $G$.
If $x, y \in G\setminus Z$,
then $\bar{S}_{G}(x,y)=\bar{S}_{G'}(x,y)$.
If $x\in Z$, then there exists only one path $P$ between
$x$ and $y$. We denote the neighbor of $x$ in $P$ by $v$. Then we have $\bar{S}_{G}(x,y)=\{v\}$. Hence, $rvd(G)\leq k+1$.

Now we have that if $rvd(G)=k$, then the size
of a connected graph $G$ of order $n$ is at least
$n+k-2$. It remains to show that for each
pair $k$, $n$ of positive integers with $1\leq k\leq n-1$,
there is a connected graph $G$ of order $n$ and size
$n+k-2$ such that $rvd(G)=k$.
We construct the graph $G_k$ as follows. For $1\leq k\leq n-2$,
given two vertices $u$ and $v$, $G_k$ is a graph obtained by
adding $k$ paths of length $2$ between $u$ and $v$ and $n-k-2$ pendent edges to $u$. Now we assign distinct colors to the $k$ common neighbors of $u$ and $v$ using colors in $[k]$, and assign color $1$ to vertex $u$, color $2$ to the remaining vertices.
For $k=n-1$, let $G_{n-1}=G_{n-2}+uv$.
Now we assign distinct colors to the $n-2$ common neighbors of $u$ and $v$ using colors in $[n-2]$,
and assign color $1$ to $u$, color $n-1$ to vertex $v$.
It is easy to verify that these colorings are
rainbow vertex-disconnection colorings of $G_k$ ($1\leq k\leq n-1$).
Thus, $rvd(G_k)\leq k$.
Furthermore, $rvd(G_k)\geq k$ by Theorem \ref{rvdlocalconn}, and so $rvd(G_k)=k$.
\end{proof}

\begin{lem}\label{rvd-minn}
For a graph $G$, if any two vertices have at least two common neighbors, then $|E(G)|\geq 2n-4+\lceil\frac{n}{2}\rceil$. Furthermore, the bound is sharp.
\end{lem}
\begin{proof}
Suppose that $G$ is an extremal such graph with minimum size.
It remains to prove that $|E(G)|\geq 2n-4+\lceil\frac{n}{2}\rceil$.
Since any two vertices have at least two common neighbors, we have $\delta(G)\geq3$. If there is exactly one vertex with degree at most $4$, then
$e(G)> \frac{5(n-1)}{2}>2n-4+\lceil\frac{n}{2}\rceil,$ a contradiction. So, there are at least two vertices of $G$ with degree at most $4$.
Therefore, there are two vertices $u,v$ of $G$, such that either $d(u)=d(v)=4$, or $d(u)=3$, $d(v)=4$, or $d(u)=d(v)=3$. Let $N(u)-v=S_1$, $N(v)-u=S_2$ and $S=S_1\cup S_2$. Let
$$\theta=
\begin{cases}
1, & if\ u\mbox{ is adjacent to }v,\\
0, & if\ u\mbox{ is nonadjacent to }v.
\end{cases}$$

Note that every vertex of $Q=V(G)-S-u-v$ has at least two vertices in $S$ and its degree is at least $3$.
Let $E_0$ be the set of edges having at least one end in $Q$. Then, $|E_0|\geq 2|Q|+\left\lceil\frac{|Q|}{2}\right\rceil$.
Therefore,
\begin{align*}
|E(G)|&\geq |E(u,S_1)|+|E(v,S_2)|+|E(G[S])|+\theta+|E_0|\\
       &\geq|S_1|+|S_2|+|E(G[S])|+\theta+2|Q|+\left\lceil\frac{|Q|}{2}\right\rceil.
\end{align*}

Let $x$ be a vertex of $G$. For $y\in N(x)$, the two common neighbors of $x$ and $y$ are in $N(x)$. So, every vertex of $N(x)$ has a degree at least two in $G[N(x)]$. Furthermore,
if $d(x)=4$, then $G[N(x)]$ contains a $4$-cycle; if $d(x)=3$, then $G[N(x)]$ is a $3$-cycle.

\textbf{Case 1}. $d(u)=d(v)=4$.

Suppose $N(u)=\{x_1,x_2,x_3,x_4\}$.

{\em Subcase $1.1.$} $|S_1\cap S_2|=2$.
\begin{figure}[h]
    \centering
    \includegraphics[width=150pt]{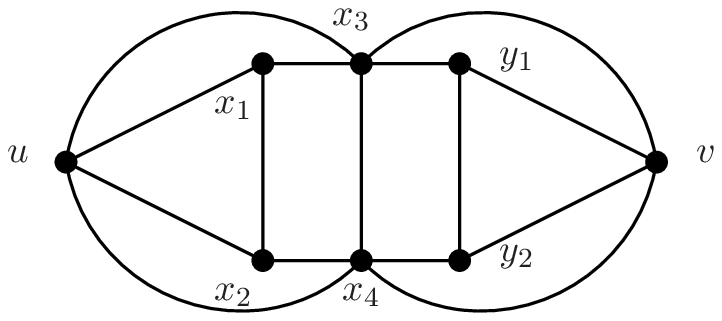}\\
    \caption{} \label{1}
\end{figure}
For $\theta=0$, $S_1=N(u)$ and $G[S_1]$ contains a $4$-cycle. By symmetry, $G[S_2]$ contains a $4$-cycle, as shown in Fig. \ref{1}.
Since $x_1,v$ ($x_2,v$) have two common neighbors in $S_2$, there is an edge connecting $x_1$ ($x_2$) to some vertex of $S_2$.
Therefore, $|E(G[S])|+\theta\geq 9$.
So, $|E(G)|\geq 2n-3+\lceil\frac{n}{2}\rceil$.

For $\theta=1$, suppose $x_4=v$ and $S_1\cap S_2=\{x_2,x_3\}$. Since $G[N(u)]$ contains a $4$-cycle, $|E(G[S_1])|\geq 2$.
By symmetry, $|E(G[S_2])|\geq 2$. Therefore, $|E(G[S])|\geq 4$, $|E(G[S])|+\theta\geq 5$.
So, $|E(G)|\geq 2n-4+\lceil\frac{n}{2}\rceil$.

{\em Subcase $1.2.$} $|S_1\cap S_2|=3$.

For $\theta=0$, since $G[S_i]$ contains a $4$-cycle $C_i$ where $i\in[2]$ and $C_1,C_2$ have at most two common edges, $|E(G[S])|+\theta\geq 6$.
So, $|E(G)|\geq 2n-3+\lceil\frac{n-1}{2}\rceil$.
For $\theta=1$, since $G[S_1]$ contains a $4$-cycle, $|E(G[S_1])|+\theta\geq3$.
So, $|E(G)|\geq 2n-3+\lceil\frac{n-1}{2}\rceil$.

{\em Subcase $1.3.$} $|S_1\cap S_2|=4$.
Obviously, $\theta=0$. Since $G[S]$ contains a $4$-cycle $C_4$, $|E(G[S])|+\theta\geq 4$.
So, $|E(G)|\geq 2n-3+\lceil\frac{n}{2}\rceil$.

\textbf{ Case 2}. $d(u)=d(v)=3$.

Assume that $|S_1\cap S_2|=2$. For $\theta=0$, since $G[S_i]$ contains a $3$-cycle $C_i$ where $i\in[2]$ and $C_1,C_2$ have one common edge, $|E(G[S])|+\theta\geq 5$.
So, $|E(G)|\geq 2n-4+\lceil\frac{n}{2}\rceil$.
For $\theta=1$, $|E(G[S_1])|+\theta\geq2$.
So, $|E(G)|\geq 2n-4+\lceil\frac{n}{2}\rceil$.
Assume $|S_1\cap S_2|=3$. Obviously, $\theta=0$. Since $G[S]$ contains a $3$-cycle $C_3$, $|E(G[S])|+\theta\geq 3$.
So, $|E(G)|\geq 2n-3+\lceil\frac{n-1}{2}\rceil$.

\textbf{Case 3}. $d(u)=4$ and $d(v)=3$.

Suppose $|S_1\cap S_2|=2$.
For $\theta=0$, similar to the proof of Subcase 1.1, $|E(G[S])|+\theta\geq 8$.
So, $|E(G)|\geq 2n-2+\lceil\frac{n-1}{2}\rceil$.
For $\theta=1$, $|E(G[S_1])|+\theta\geq4$.
So, $|E(G)|\geq 2n-3+\lceil\frac{n-1}{2}\rceil$.
Suppose that $|S_1\cap S_2|=3$. Obviously, $\theta=0$. We have $|E(G[S])|+\theta\geq 5$.
So, $|E(G)|\geq 2n-3+\lceil\frac{n}{2}\rceil$.

Above all, $|E(G)|\geq 2n-4+\lceil\frac{n}{2}\rceil$.
Furthermore, we prove that the bound is sharp.
Let $H$ be a graph by adding $\lceil \frac{n-2}{2}\rceil$ edges to $G_{n-1}$ (mentioned in the proof of Lemma \ref{rvd-min1}) such that each component of
$G[V(G)\setminus\{u,v\}]$ is a $P_2$ or $P_3$ and at most one component is $P_3$. Then $|E(H)|=2n-3+\lceil \frac{n-2}{2}\rceil=2n-4+\lceil \frac{n}{2}\rceil$.
Observe that any two vertices of the graph $H$ have at least two common neighbors.
Thus, $H$ is the graph attaining the bound.
\end{proof}

By Theorem \ref{rvdn}, and Lemmas \ref{rvd-min1} and \ref{rvd-minn}, we have the following result.

\begin{thm}\label{rvd-min}
For integers $k$ and $n$ with $1\leq k\leq n$,
the minimum size of a connected graph $G$ of order
$n\geq 4$ with $rvd(G)=k$ is
$$|E(G)|_{min}=
\begin{cases}
n+k-2, & 1\leq k\leq n-1,\\
2n-4+\lceil\frac{n}{2}\rceil, & k=n.
\end{cases}$$
\end{thm}

Next, it is natural to consider the following extremal problem:

For a given pair $k$, $n$ of positive integers with $1\leq k\leq n$, what is the maximum possible size of a connected graph $G$ of order $n$ such that the rainbow vertex-disconnection number of $G$ is $k$ ?

However, for this problem we can only get the lower and upper bounds on the maximum size of a connected graph $G$ of order $n$ with $rvd(G)=k$.
We now present two known lemmas which we will be used in our proof.

\begin{lem}\cite{B}\label{rvd-lmax1}
Let $k=2,3$ and let $G$ be a graph of order $n$ such that $\kappa^+(G)\leq k$. Then
$|E(G)|\leq \lfloor \frac{k+1}{2}(n-1)\rfloor$.
\end{lem}

\begin{lem}\cite{M}\label{rvd-lmax3}
Let $G$ be a graph of order $n$. Then for $k\geq 4$,
$\max\{|E(G)|:\kappa^+(G)\leq k\}\leq k(n-1)-{k\choose2}$.
\end{lem}

From these lemmas we have the following results.

\begin{thm}\label{rvd-max1}
For $k=2$, $3$, let $G$ be a graph of order $n$ with $rvd(G)=k$. Then, $|E(G)|_{max}=\lfloor \frac{k+1}{2}(n-1)\rfloor$.
\end{thm}

\begin{proof}
Since $rvd(G)=k$, we have $\kappa^+(G)\leq k$ by Theorem \ref{rvdlocalconn}. Then we get that the maximum size of a graph $G$ of order $n$ with $rvd(G)=k$ is no more than
$\lfloor \frac{k+1}{2}(n-1)\rfloor$ by Lemma \ref{rvd-lmax1}. Furthermore, for $k=2$,
let $\mathcal{H}$ be the set of connected graphs whose blocks are triangles with the exception that at most one block is a cut edge or $C_4$.
Then for any graph $G\in \mathcal{H}$, we have that  $|E(G)|=\lfloor\frac{3}{2}(n-1)\rfloor$ and $rvd(G)=2$.
For $k=3$, we know that $rvd(W_n)=3$ for $4\mid n$ and $|E(W_n)|=2n=\lfloor2(|V(W_n)|-1)\rfloor$.
\end{proof}

\begin{thm}\label{rvd-max3}
For $k\geq 4$, let $G$ be a graph of order $n$ with $rvd(G)=k$. Then, $\frac{1}{2}k(n-1)-{k\choose2}\leq |E(G)|_{max}\leq k(n-1)-{k\choose2}$.
\end{thm}

\begin{proof}
Similar to the proof of Theorem \ref{rvd-max1}, we get $|E(G)|_{max}\leq k(n-1)-{k\choose2}$ by Theorem \ref{rvdlocalconn} and Lemma \ref{rvd-lmax3}.
For the lower bound, let $G$ be a graph such that each block is a $K_k$ except that at most one block is $K_t$, where $t=n-(k-1)\lfloor\frac{n-1}{k-1}\rfloor$.
Obviously, $rvd(G)=k$. Thus, $|E(G)|_{max}\geq \lfloor\frac{n-1}{k-1}\rfloor \times|E(K_k)|+|E(K_t)|\geq\frac{1}{2}k(n-1)-{k\choose2}$.
\end{proof}

As one can see, further efforts are needed to get the exact value of $|E(G)|_{max}$ for $k\geq 4$.

\end{document}